\documentclass[12pt,a4paper]{amsart}
\usepackage{amssymb,amsmath}
\usepackage{hyperref}

\numberwithin{equation}{section}
\newtheorem{theorem}{Theorem}[section]
\newtheorem{lemma}[theorem]{Lemma}
\newtheorem{proposition}[theorem]{Proposition}
\newtheorem{corollary}[theorem]{Corollary}

\theoremstyle{definition}
\newtheorem{definition}[theorem]{Definition}

\theoremstyle{remark}

\newcommand\Supp{\operatorname{Supp}}
\newcommand\coker{\operatorname{coker}}

\newcommand\Ann{\operatorname{Ann}}

\newcommand\Tor{\operatorname{Tor}}
\newcommand\Hom{\operatorname{Hom}}

\newcommand\Ext{\operatorname{Ext}}
\newcommand\Rad{\operatorname{Rad}}

\newcommand\grade{\operatorname{grade}}
\newcommand\Cograde{\operatorname{Cograde}}

\newcommand{\qism}{\stackrel{\sim}{\longrightarrow}}

\begin{document}
\title[Generalized completion homology modules]{On generalized completion homology modules}%
\author[W. Mahmood]{Waqas Mahmood }
\address{Quaid-I-Azam University Islamabad, Pakistan}%
\email{ waqassms$@$gmail.com}

\thanks{This research was partially supported by Higher Education Commission, Pakistan}

\maketitle
\begin{abstract} Let $I$ be an ideal of a commutative Noetherian ring $R$. Let $M$ and $N$ be any $R$-modules. The generalized completion homology modules $L_i\Lambda^I (N,M)$ are defined, for $i\in \mathbb{Z}$, as the homologies of the complex $\lim\limits_{\longleftarrow}(N/I^sN\otimes_R F_{\cdot}^R)$. Here $F_{\cdot}^R$ denotes a flat resolution of $M$. In this article we will prove the vanishing and non-vanishing properties of $L_i\Lambda^I (N,M)$. We denote  $H^{i}_{I}(N,M)$ (resp. $U^I_i(N,M)$) by the generalized local cohomology modules (resp. the generalized local homology modules). As a technical tool we will construct several natural homomorphisms of $L_i\Lambda^I (N,M)$, $H^{i}_{I}(N,M)$ and $U^I_i(N,M)$. We will investigate when these natural homomorphisms are isomorphisms. Moreover if $M$ is Artinian and $N$ is finitely generated then it is proven that $L_i\Lambda^I (N,M)$ is isomorphic to $U^I_i(N,M)$ for each $i\in \mathbb{Z}$. The similar result is obtained for $H^i_{I}(N,M)$. Furthermore if both $M$ and $N$ are finitely generated with $c=\grade(I,M)$. Then we are able to prove several necessary and sufficient conditions such that $H^i_{I}(M)=0$ for all $i\neq c.$ Here $H^i_{I}(M)$ denotes the ordinary local cohomology module.
\end{abstract}

\section{Introduction}
Let $R$ be a commutative Noetherian ring, $I$ an ideal of $R$ and $M$ an $R$-module. The left derived functors, $L_i\Lambda^I (M)$, of the $I$-adic completion of $M$ was firstly studied by Matlis (see \cite{mat} and \cite{mat1} for details). Suppose that $F_{\cdot}^R$ is a flat resolution of $M$ and $N$ an arbitrary $R$-module. In \cite{n} T. Nam introduced the notion of the generalized completion homology modules $L_i\Lambda^I (N,M)$, for $i\in \mathbb{Z}$, as the homologies of the complex $\lim\limits_{\longleftarrow}(N/I^sN\otimes_R F_{\cdot}^R)$. It is clear that if $N=R$ then $L_i\Lambda^I (R,M)=L_i\Lambda^I (M)$ are the well-known left derived functors of completion.

In this article we will study various properties of $L_i\Lambda^I (N,M)$. In general these functors are difficult to compute. Firstly, in \cite[Theorem 20]{mat1}, Matlis has investigated some results on $L_i\Lambda^I (M)$. In particular if $I$ is generated by a regular sequence then he proved the following isomorphism
\[
L_i\Lambda^I (M)\cong \Ext_R^i(\lim_{\longrightarrow}R/I^s,M)
\]
for all $i\in \mathbb{Z}$. After that Greenlees and May provided a criterion to compute $L_i\Lambda^I (M)$ in terms of certain local homology groups of $M$ (see \cite[Thoerem 2.2]{g}). Most recently P. Schenzel has proved (see \cite[Theorem 1.1]{pet2}) that we can construct the completion homology modules as a dual to the local cohomology modules $H^i_{I}(M)$. We refer to \cite{goth} for the definition of local cohomology modules.

Moreover Cuong and Nam introduced the local homology modules $U^I_i(M)$ which is actually a dual definition of the local cohomology modules. They proved, in \cite[Propsition 4.1]{c1}, that if $M$ is an Artinian $R$-module then $U^I_i(M)$ is isomorphic to $L_i\Lambda^I (M)$ for all $i\in \mathbb{Z}$. Moreover a duality between local homology and local cohomology modules was proved in their paper.

Few years ago Herzog defined, in \cite{her}, the generalization of local cohomology modules as follows:
\[
H^i_{I}(N,M)\cong \lim\limits_{\longrightarrow} \Ext_R^i(N/I^sN,M) \text{ for all $i\in \mathbb{Z}$. }
\]
Many people have studied about the modules $H^i_{I}(N,M)$. For instance the vanishing, non-vanishing and Artinianess properties of $H^i_{I}(N,M)$. For details we refer the reader to \cite{l}, \cite{c2}, \cite{her1}, \cite{k1} and \cite{ma}. Moreover Yassemi defined the functor $\Gamma_I(M,N)$ and proved that its cohomologies are isomorphic to Herzog's generalized local cohomology functor (see \cite[Theorem 3.4]{y}). As a dualization of this T. Nam introduced the generalized local homology modules $U^I_i(N,M)$ which are defined as $\lim\limits_{\longleftarrow}\Tor^R_{i}(N/I^sN, M)$ (see \cite{n1}).

Bijan-Zadeh and Moslehi have proved finiteness and vanishing properties of $U^I_i(N,M)$ (see \cite[Theorems 3.1 and 4.4]{b2}). Moreover the non-vanishing results are obtained in \cite[Theorem 2.4]{k}.

Our main purpose is to investigate the vanishing and non-vanishing results of $L_i\Lambda^I (N,M)$. It seems to be not so much known about the highest natural number $s\in \mathbb{N}$ such that $H^{s}_{I}(N,M)\neq 0$. In \cite[Theorem 3.5]{d1} Divaani-Aazar and Hajikarimi proved that we can compute such $s$ for a Cohen-Macaulay local ring $R$ with either projective dimension of $N$ or injective dimension of $M$ is finite.

If $H^{i}_{I}(N,M)=0$ for all $i\neq c=\grade(I,M)$. Then we call $M$ a cohomologically complete intersection with respect to the pair $(N,I)$. In case of $N=R$ we will say that $M$ is cohomologically complete intersection with respect to $I$.

Further assume that $R$ is local and $D(\cdot)=\Hom_R(\cdot,E_R(k))$ denotes the ordinary Matlis dual functor. It is shown, in Corollary \ref{0.01}, that $M$ is a cohomologically complete intersection with respect to $(N,I)$ if and only if $L_i\Lambda^I (N,D(M))=0$ for all $i\neq c$ if and only if $U^I_i(N,D(M))=0$ for all $i\neq c$.

As a technical tool we will construct several natural homomorphisms of $L_i\Lambda^I (N,M)$, $H^{i}_{I}(N,M)$ and $U^I_i(N,M)$. With the assumption of $M$ is cohomologically complete intersection with respect to $I$ we will show that these homomorphisms are isomorphisms. In this regard we prove the following Theorems:

\begin{theorem}\label{22a}
Let $0\neq M$ be a module over an arbitrary commutative Noetherian ring $R$. Let $I$ be an ideal and $c=\grade(I,M)$. Then the following conditions are equivalent:
\begin{itemize}
\item[(i)] $M$ is cohomologically complete intersection with respect to $I$.

\item[(ii)] For any $R$-module $N$ the natural homomorphism
\[
H^{i}_{I}(N,H^c_{I}(M))\to H^{i+c}_{I}(N,M)
\]
is an isomorphism for all $i\in \mathbb{Z}$.
\end{itemize}
Assume in addition that $R$ is local. Then the above conditions are equivalent to the following:
\begin{itemize}
\item[(iii)] For any finitely generated $R$-module $N$ the natural homomorphism
\[
U_{i+c}^I(N, D(M))\to U_{i}^I(N, D(H^c_I(M)))
\]
is an isomorphism for all $i\in \mathbb{Z}$.

\item[(iv)] For any finitely generated $R$-module $N$ the natural homomorphism
\[
L_{i+c}\Lambda^I(N, D(M))\to L_{i}\Lambda^I(N, D(H^c_I(M)))
\]
is an isomorphism for all $i\in \mathbb{Z}$.
\end{itemize}
\end{theorem}

\begin{theorem}\label{221a}
Let $0\neq M$ be a module over a local ring $R$. Let $I$ be an ideal and $c=\grade(I,M)$. Then the following conditions are equivalent:
\begin{itemize}
\item[(i)] $M$ is cohomologically complete intersection with respect to $I$.

\item[(ii)] For any $R$-module $N$ the natural homomorphism
\[
H^i_{I}(N,D(M))\to H^{i+c}_{I}(N,D(H^c_I(M)))
\]
is an isomorphism for all $i\in \mathbb{Z}$.
\end{itemize}
Assume in addition that $M$ is Artinian. Then the above conditions are equivalent to the following:
\begin{itemize}
\item[(iii)] For any finitely generated $R$-module $N$ the natural homomorphism
\[
U_{i+c}^I(N, D(D(H^c_I(M))))\to U_{i}^I(N, M)
\]
is an isomorphism for all $i\in \mathbb{Z}$.

\item[(vi)] For any finitely generated $R$-module $N$ the natural homomorphism
\[
L_{i+c}\Lambda^I(N, D(D(H^c_I(M))))\to L_{i}\Lambda^I(N, M)
\]
is an isomorphism for all $i\in \mathbb{Z}$.
\end{itemize}
\end{theorem}

The natural homomorphisms of Theorems \ref{22a} and \ref{221a} are derived in Theorem \ref{2} and Corollary \ref{2.11}. In case of $N$ is finitely generated and $M$ is Artinian we are succeeded to prove that $L_{i}\Lambda^I(N, M)$ is isomorphic to $U_{i}^I(N, M)$ for all $i\in \mathbb{Z}$ (see Proposition \ref{2.3}). This actually generalizes the result of \cite[Propsition 4.1]{c1}. As a consequence of this we will prove that the sequence of functors $L_{i}\Lambda^I(N, M)$ is positive strongly connected (see Corollary \ref{1.1.1}). In Corollary \ref{31}, as an application of our main Theorem \ref{22a}, there are several characterization of grade and co-grade. For the definition and basic results of co-grade we refer to \cite{o}.

\section{Generalized homologies and cohomologies}

In this section we will fix the notations which will be used in the sequel of the paper. We will denote $R$ by a commutative Noetherian ring. Let $f: X\to Y$ be a morphism of complexes of $R$-modules. If the map $H^i(X)\to H^i(Y)$ induced by $f$ is an isomorphism for each $i\in \mathbb{Z}$. Then the $f$ is called a quasi-isomorphism. For our convenience we will write it as $f:X \qism Y$. Moreover for well-known results about homological algebra we recommend the reader to consult \cite{b} , \cite{har1} and \cite{w}.

Moreover note that throughout the paper we will deal with cochain complexes because it helps in dealing with indices issues.

Let $M$ and $N$ be arbitrary $R$-modules. For an ideal $I$ of $R$ let $H^i_I(M)$, $i\in \mathbb Z$, denote the local cohomology modules of $M$ with respect to $I$ (see \cite{goth} for its definition). In \cite{her} Herzog introduced the generalized local cohomology modules $H^i_{I}(N,M)$ as the direct limit of the direct system $\{\Ext_R^i(N/I^sN,M):i\in \mathbb{Z} \}$. Moreover T. Nam defined the generalized local homology modules $U^I_i(N,M)$ as the inverse limit of the inverse system $\{\Tor^R_{i}(N/I^sN, M):i\in \mathbb{Z} \}$ (see \cite{n1}). For a flat resolution $F_{\cdot}^R$ of $M$ T. Nam introduced the notion of the generalized completion homologies as
\[
L_i\Lambda^I (N,M):= H_i(\lim\limits_{\longleftarrow}(N/I^sN\otimes_R F_{\cdot}^R)) \text { for all $i\in \mathbb{Z}$. }
\]
((see \cite{n})). Note that $L_i\Lambda^I (N,M)$ are independent of the choice of a flat resolution of $M$. Since the tensor product is not left exact and the inverse limit is not right exact. Therefore, in general,
\[
L_0\Lambda^I (N,M)\neq \lim\limits_{\longleftarrow}(N/I^sN\otimes_R M).
\]
Clearly if $N=R$ then $L_i\Lambda^I (R,M)=L_i\Lambda^I (M)$ is the usual left derived functors of the completion. For more details about $L_i\Lambda^I (M)$ one should see \cite{g}.

Note that if $R$ is a local ring then we will consider the unique maximal ideal as $\mathfrak{m}$. Let $E= E_R(k)$ be the injective hull of the residue field $k=R/\mathfrak{m}$. Also $D(\cdot)=\Hom_R(\cdot,E)$ stands for the Matlis dual functor.

\begin{definition}
Let $I$ be an ideal of an arbitrary ring $R$ and $M$ an $R$-module such that $IM\neq M$. Then grade of $M$ is defined as:
\[
\grade(I,M)= \inf\{i\in \mathbb{Z}: H^i_I(M)\neq 0\}.
\]
\end{definition}

In the following we need the definition of cograde which is actually defined in \cite[Definition 3.10]{o}.

\begin{definition}
Let $N$ be an $R$-module. Then an element $x\in R$ is called co-regular if $\Ann_N(xR)\neq 0$. Moreover a sequence $\underline{x}=x_1,\cdots,x_r\in R$ is called co-regular sequence if the following conditions hold:
\begin{itemize}
\item[(i)] $\Ann_N(\underline{x}R)\neq 0$.

\item[(ii)] Each $x_i$ is an $\Ann_N((x_1,\cdots,x_{i-1})R)$-coregular element for all $i=1,\cdots, r$.
\end{itemize}
\end{definition}

Suppose that $N$ is a finitely generated $R$-module and $M$ is an Artinian $R$-module. Then the length of any maximal $M$-coregular sequence contained in $\Ann_R(N)$ is called $\Cograde_M(N)$. Note that it is well-defined see \cite[Definition 3.10]{o}.

\begin{definition}
Let $R$ be any ring (not necessarily local) and $I$ an ideal of $R$. Suppose that $M$ and $N$ are $R$-modules with $c=\grade(I,M)$. Then we say that $M$ is cohomologically complete intersection with respect to the pair $(N,I)$ if $H^{i}_{I}(N,M)=0$ for all $i\neq c$. Moreover for $N=R$ we call $M$ a cohomologically complete intersection with respect to $I$.
\end{definition}

In this paper we will investigate that $M$ is cohomologically complete intersection with respect to $I$. 

Since $ U^I_i(N,M)= \lim\limits_{\longleftarrow}\Tor^R_{i}(N/I^sN, M)$. So it naturally induces the following homomorphisms
\[
L_i\Lambda^I (N,M)\to U^I_i(N,M)
\]
for all $i\in \mathbb{Z}$. As a first result of the generalized completion homologies we will prove that these natural homomorphisms are surjective. Note the following Lemma comes from \cite[Proposition 1.1]{g}.

\begin{lemma} \label{1.3}
Let $R$ be an arbitrary ring and $I$ an ideal of $R$. Then for any $R$-modules $M$ and $N$ there is an exact sequence
\[
0\to \lim_{\longleftarrow}^1 \Tor^R_{i+1}(N/I^sN, M)\to L_i\Lambda^I (N,M)\to U^I_i(N,M)\to 0
\]
for each $i\in \mathbb{Z}$.
\end{lemma}

\begin{proof}
It can be easily proved by following the similar steps of \cite[Proposition 1.1]{g}. We left it to the readers.
\end{proof}

In the sequel of the paper we will need the following Hom-Tensor Duality.

\begin{lemma}\label{445}
For a local ring $R$ let $M$ and $N$ be any $R$-modules. Then for each $i\in \mathbb{Z}$ the following hold:
\begin{itemize}
\item[(i)] $\Ext^{i}_R(N,D(M))\cong D(\Tor_{i}^R(N, M))$.

\item[(ii)] If $N$ is finitely generated then
\[
D(\Ext^{i}_R(N,M))\cong \Tor_{i}^R(N, D(M)).
\]
\end{itemize}
\end{lemma}

\begin{proof}
For the proof see \cite[Example 3.6]{h}.
\end{proof}

In the next result we will prove that the above natural homomorphism $L_i\Lambda^I (N,M)\to U^I_i(N,M)$ is an isomorphism for each $i\in \mathbb{Z}$ under the additional assumption of $M$ is an Artinian $R$-module and $N$ is a finitely generated $R$-module. Note that a similar result is obtained in \cite[Theorems 3.2 adn 3.6]{n}. It is well-known that if $N=R$ then this homomorphism is an isomorphism (see \cite[Propsition 4.1]{c1}).

\begin{proposition}\label{2.3}
Let $I$ be an ideal over an arbitrary ring $R$. Let $N$ be a finitely generated $R$-module. Then we have:
\begin{itemize}
\item[(i)] If $M$ is an Artinian $R$-module then the natural homomorphism
\[
L_i\Lambda^I (N,M)\to U^I_i(N,M)
\]
is an isomorphism for each $i\in \mathbb{Z}$. Moreover if $R$ is local then $D(H^i_{I}(N,D(M)))\cong U^I_i(N,M)$ for each $i\in \mathbb{Z}$.

\item[(ii)] If $M$ is any $R$-module and $R$ is local. Then the natural homomorphism
\[
L_i\Lambda^I (N,D(M))\to U^I_i(N,D(M))
\]
is an isomorphism for each $i\in \mathbb{Z}$. Moreover $D(H^i_{I}(N,M))\cong U^I_i(N,D(M))$ for each $i\in \mathbb{Z}$.
\end{itemize}
\end{proposition}

\begin{proof}
To prove the statement in $(i)$ we will follow the technique of \cite[Proposition 4.1]{c1}. Since $N$ is a finitely generated module and $M$ is an Artinian module over a Noetherian ring $R$. Then it follows that $\Tor^R_{i+1}(N/I^sN, M)$ is an Artinian $R$-module for all $i\in \mathbb{Z}$. Therefore $\lim\limits_{\longleftarrow}^1 \Tor^R_{i+1}(N/I^sN, M)=0$ for all $i\in \mathbb{Z}$. This proves that the natural homomorphism
\[
L_i\Lambda^I (N,M)\to U^I_i(N,M)
\]
is an isomorphism in view of Lemma \ref{1.3}. Note that the Hom functor transforms the direct systems into the inverse systems in the first variable. So by Lemma \ref{445} we have
\[
D(H^i_{I}(N,D(M)))\cong \lim_{\longleftarrow}\Hom_R(\Ext_R^{i}(N/I^sN, D(M)),E)\cong \lim_{\longleftarrow} \Tor^R_{i}(N/I^sN, M)
\]
for all $i\in \mathbb{Z}$. To this end note that $D(D(M))\cong M$ sine $M$ is Artinian. This finishes the proof of $(i)$.

Now we prove the statement in $(ii)$. To do this suppose that $M$ is any $R$-module and $R$ is local. By definition of the direct limit there is a short exact sequence
\[
0\to \bigoplus_{s\in \mathbb{N}} \Ext_R^{i}(N/I^sN, M)\to \bigoplus_{s\in \mathbb{N}} \Ext_R^{i}(N/I^sN, M)\to H^i_I(N,M)\to 0
\]
for each $i\in \mathbb{Z}$. Since $N$ is finitely generated so the application of the Matlis dual functor to the last sequence induces the following exact sequence
\[
0\to D(H^i_I(N,M))\to \prod_{s\in \mathbb{N}} \Tor^R_{i}(N/I^sN, D(M))\mathop\to\limits^{\Psi_i} \prod_{s\in \mathbb{N}} \Tor^R_{i}(N/I^sN, D(M))\to 0
\]
To this end note that $D(\Ext_R^{i}(N/I^sN, M))\cong \Tor^R_{i}(N/I^sN, D(M))$ (see Lemma \ref{445}) and it transforms the direct system $\{\Ext_R^{i}(N/I^sN, M):i\in \mathbb{N}\}$ into the following inverse system
\[
\{\Tor^R_{i}(N/I^sN, D(M)):i\in \mathbb{N}\}.
\]
Now by \cite[Definition 3.5.1]{w} it follows that
\[
\lim_{\longleftarrow} \Tor^R_{i}(N/I^sN, D(M))\cong D(H^i_I(N,M)) \text { and } \lim_{\longleftarrow}^1 \Tor^R_{i}(N/I^sN, D(M))\cong \coker \Psi_i=0
\]
for all $i\in \mathbb{Z}$. On the other side by Lemma \ref{1.3} we have the following short exact sequence
\[
0\to \lim_{\longleftarrow}^1 \Tor^R_{i+1}(N/I^sN, D(M))\to L_i\Lambda^I (N,D(M))\to U^I_i(N,D(M))\to 0
\]
Hence the homomorphism $L_i\Lambda^I (N,D(M))\to U^I_i(N,D(M))$ becomes an isomorphism for all $i\in \mathbb{Z}$. This completes the proof of the Proposition.
\end{proof}

The next Corollary shows that the sequence of the functors $\{L_i\Lambda^I (N,M):i\in \mathbb{Z}\}$ is positive strongly connected on the category of Artinian $R$-modules (see \cite[p. 212]{r}).

\begin{corollary}\label{1.1.1}
Suppose that $I$ is an ideal of a local ring $R$. Consider the following short exact sequence of any $R$-modules
\[
0\to M_1\to M_2\to M_3\to 0.
\]
Then for a finitely generated $R$-module $N$ we have the following results:
\begin{itemize}
\item[(i)] There is a long exact sequence of generalized completion homologies as follows:
\begin{gather*}
\cdots \to L_{i}\Lambda^I(N, D(M_3))\to L_{i}\Lambda^I(N, D(M_2))\to L_{i}\Lambda^I(N, D(M_1))\to\\
L_{i-1}\Lambda^I(N, D(M_3))\to \cdots \to L_{0}\Lambda^I(N, D(M_2))\to L_{0}\Lambda^I(N, D(M_1))\to 0
\end{gather*}
\item [(ii)] Suppose that $M_i$'s is Artinian for each $i=1,2,3$. Then there is a long exact sequence of generalized completion homologies as follows:
\begin{gather*}
\cdots \to L_{i}\Lambda^I(N, M_1)\to L_{i}\Lambda^I(N, M_2)\to L_{i}\Lambda^I(N, M_3)\to\\
L_{i-1}\Lambda^I(N, M_1)\to \cdots \to L_{0}\Lambda^I(N, M_2)\to L_{0}\Lambda^I(N, M_3)\to 0
\end{gather*}
\end{itemize}
\end{corollary}

\begin{proof}
It is an immediate consequence of the fact that $L_{i}\Lambda^I(N, M)$ is the $i$th left derived functor of the complex $\lim\limits_{\longleftarrow}(N/I^sN\otimes_R F_{\cdot}^R)$. Here $F_{\cdot}^R$ denotes a flat resolution of $M$.
\end{proof}

In order to describe the next result we will need a preparation of Koszul complex. Let $x_1,\ldots, x_r \in I$ be a generating system of $I$. Now  consider the Koszul complex $K^{\cdot}_{\underline{x}^t}$ with respect to ${\underline{x}^t}=x_1^t,\ldots, x_r^t$ (for more details about Koszul complex see \cite{m}). Suppose that $F_{\cdot}^R$ denote a projective resolution of a finitely generated $R$-module $N$. Let $C^{\cdot}_{t}$ be the total complex
associated to the double complex $K^{\cdot}_{\underline{x}^t}\otimes_R F_{\cdot}^R$. By \cite[Satz 1.1.6]{her} for any $R$-module $M$ there are the isomorphisms
\[
H^{i}_I(N, M)\cong \lim_{\longrightarrow}H^i(\Hom_R(C^{\cdot}_{t},M)).
\]
In the following we will relate the generalized completion homologies $L_i\Lambda^I (N,M)$ with the inverse limits of the homologies of the dual complex $C^{\cdot}_{t}\otimes_R M$. Note that the next result is a consequence of \cite[Theorem 5.1]{b2}.

\begin{lemma}
Let $R$ be a local ring and $I$ an ideal of $R$. Let $N$ be a finitely generated $R$-module and $M$ an Artinian $R$-module. Then there is an isomorphism
\[
L_i\Lambda^I (N,M)\cong \lim_{\longleftarrow}H^i(C^{\cdot}_{t}\otimes_R M)
\]
for all $i\in \mathbb{Z}$.
\end{lemma}

\begin{proof}
For the proof we refer to \cite[Theorem 5.1]{b2} and Proposition \ref{2.3}.
\end{proof}

\subsection{Some natural homomorphisms}
In this section we will derive some natural homomorphisms of the modules $H^i_{I}(N,M)$,  $U^I_i(N,M)$ and $L_i\Lambda^I (N,M)$. Before this we recall the definition of the truncation complex which was introduced in \cite[Definition 4.1]{p2}. Let $E^{\cdot}_R(M)$ be a minimal injective resolution of an $R$-module $M$ and $\grade (I,M)= c$. Note that
\[
E^{\cdot}_R(M)^i\cong\bigoplus\limits_{{\mathfrak p}\in \Supp M}E_R(R/{\mathfrak p})^{\mu_{i}({\mathfrak p}, M)},
\]
where $\mu_{i}({\mathfrak p}, M)=\dim_{k({\mathfrak p})}(\Ext^i_{R_{\mathfrak p}}(k({\mathfrak p}),M_{\mathfrak p}))$.
Since $\Gamma_{I}(E_R(R/{\mathfrak p}))=0$ for all $\mathfrak p\notin V(I)$ and $\Gamma_{I}(E_R(R/{\mathfrak p}))= E_R(R/{\mathfrak p})$ for all $\mathfrak p\in V(I)$. It follows that for all $i< c$ we have
\[
\Gamma_I(E^{\cdot}_R(M))^{i}= 0.
\]
Then there is a natural embedding of complexes $H^c_I(M)[-c]\rightarrow \Gamma_I(E^{\cdot}_R(M))$. The cokernel $C^{\cdot}_M(I)$ of this embedding is called the truncation complex of $R$ with respect to $I$. Therefore there is an exact sequence of complexes
\[
0\to H^c_I(M)[-c]\to \Gamma_I(E^{\cdot}_R(M))\to C^{\cdot}_M(I)\to 0.
\]
Moreover note that $H^i(C^{\cdot}_M(I))=0$ for all $i\leq c$ and $H^i(C^{\cdot}_M(I))\cong  H^i_I(M)$ for all $i> c$.

Let $\underline{x}= x_1, \ldots ,x_r\in I$ denote a system of elements of $R$ such that $\Rad I= \Rad(\underline{x})R$.
We consider the \v{C}ech complex $\Check{C}_{\underline{x}}$ with respect to $\underline{x}= x_1,...,x_r$. That is
\[
\Check{C}_{\underline{x}}= \mathop\otimes\limits_{i}^r \Check{C}_{x_i},
\]
where $\Check{C}_{x_i}$ is the complex $0\rightarrow R\rightarrow R_{x_i}\rightarrow 0$.

As a first application of the truncation complex we will prove the following result:

\begin{theorem}\label{2}
Let $R$ be any ring and $M$ an $R$-module. Let $I$ be an ideal of $R$ with $\grade(I,M)=c.$ Then for any $R$-module $N$ the following conditions hold:
\begin{itemize}
\item[(i)] There are the natural homomorphisms
\[
H^i_I(N,H^c_I(M))\to H^{i+c}_I(N, M)
\]
for all $i\in \mathbb{Z}$. Moreover these are isomorphisms for all $i\in \mathbb{Z}$ if and only if $H^i_I(N,C^{\cdot}_M(I))= 0$ for all $i\in \mathbb{Z}$.

\item[(ii)] Suppose that $R$ is local. Then there are the natural homomorphisms
\[
H^i_{I}(N,D(M))\to H^{i+c}_{I}(N,D(H^c_I(M)))
\]
for all $i\in \mathbb{Z}$. Moreover these are isomorphisms if and only if $H^i_{I}(N,D(C^{\cdot}_M(I)))=0$ for all $i\in \mathbb{Z}.$
\end{itemize}
\end{theorem}

\begin{proof}
Let $F_{\cdot}(N/I^sN)$ be a free resolution of $N/I^sN$ for $s\in \mathbb{N}$. Apply the functor $\Hom_R(F_{\cdot}(N/I^sN), .)$ to the short exact sequence of the truncation complex. Then it induces the following short exact sequences of complexes of $R$-modules
\begin{gather*}
0\rightarrow \Hom_R(F_{\cdot}(N/I^sN), H^c_I(M))[-c]\rightarrow \Hom_R(F_{\cdot}(N/I^sN), \Gamma_I(E^{\cdot}_R(M)))\rightarrow \\
\Hom_R(F_{\cdot}(N/I^sN), C^{\cdot}_M(I))\rightarrow 0.
\end{gather*}
Let us investigate the cohomologies of the complex in the middle. To this end let us denote $X:=\Hom_R(F_{\cdot}(N/I^sN), \Gamma_I(E^{\cdot}_R(M)))$. Since the functor $\Gamma_I$ sends injective modules to injective modules so it follows that $H^i(X)\cong H^i(\Hom_R(N/I^sN, \Gamma_I(E^{\cdot}_R(M))))$ for all $i\in \mathbb{Z}$ and for all $s\in \mathbb{N}$.

Now $\Supp_R(N/I^sN)\subseteq V(I)$ so by \cite[Lemma 2.2]{p3} there is an isomorphism of complexes
\[
\Hom_R(N/I^sN, \Gamma_I(E^{\cdot}_R(M)))\cong \Hom_R(N/I^sN, E^{\cdot}_R(M))
\]
It follows that $H^i(X)\cong H^i(\Hom_R(F_{\cdot}(N/I^sN), E^{\cdot}_R(M)))\cong \Ext^i_R(N/I^sN, M)$ for all $i\in \mathbb{Z}$ and for all $s\in \mathbb{N}$. Then the long exact sequence of cohomologies induces the exact sequence
\begin{equation}\label{1a}
\Ext^{i-c}_R(N/I^sN, H^c_I(M))\to \Ext^{i}_R(N/I^sN, M)\to \Ext^{i}_R(N/I^sN, C^{\cdot}_M(I))
\end{equation}
for all $i\in \mathbb{Z}$ and for all $s\in \mathbb{N}$. Since the direct limit is an exact functor. So by passing to the direct limit of the last sequence provides the natural homomorphisms in $(i)$. Moreover it is easy to see that these homomorphisms becomes isomorphisms if and only if $H^i_I(N,C^{\cdot}_M(I))= 0$ for all $i\in \mathbb{Z}$.

To construct the natural homomorphisms of the statement in $(ii)$ we firstly apply the Matlis dual functor to the truncation complex. Then we get the exact sequence
\[
0\to D(C^{\cdot}_M(I))\to D(\Gamma_I(E^{\cdot}_R(M)))\to D(H^c_I(M))[c]\to 0.
\]
Then the last sequence induces the following exact sequence
\begin{gather*}
0\to \Hom_R(F_{\cdot}(N/I^sN),D(C^{\cdot}_M(I)))\to \Hom_R(F_{\cdot}(N/I^sN),D(\Gamma_I(E^{\cdot}_R(M))))\to \\ \Hom_R(F_{\cdot}(N/I^sN),D(H^c_I(M)))[c] \to 0.
\end{gather*}
We are interested in the cohomologies of $X=\Hom_R(F_{\cdot}(N/I^sN),D(\Gamma_I(E^{\cdot}_R(M))))$. First of all note that there is an isomorphism of complexes
\[
X\cong D(F_{\cdot}(N/I^sN)\otimes_{R}\Gamma_I(E^{\cdot}_R(M))).
\]
(see \cite[Proposition 5.15]{har1}). Since the Matlis dual functor $D(\cdot)$ is exact and cohomologies commutes with exact functor. So the last isomorphism induces that
\[
H^i(X)\cong D(H^{-i}(F_{\cdot}(N/I^sN)\otimes_{R}\Gamma_I(E^{\cdot}_R(M))))
\]
for all $i\in \mathbb{Z}$. In order to compute the cohomologies of $X$ we have to calculate the cohomologies of $Y:=F_{\cdot}(N/I^sN)\otimes_{R}\Gamma_I(E^{\cdot}_R(M))$. Note that $E^{\cdot}_R(M)$ is a complex of injective $R$-modules. Then by \cite[Theorem 3.2]{pet2} we have $Y\qism F_{\cdot}(N/I^sN)\otimes_{R} \Check{C}_{\underline{y}}\otimes_{R} E^{\cdot}_R(M)$.
Here $\Check{C}_{\underline{y}}$ denotes the \v{C}ech complex with respect $\underline{y} = y_1,\ldots, y_r\in I$ such that $\Rad(IR)=\Rad(\underline{y}R).$

Since tensoring with the right bounded complexes of flat $R$-modules preserves quasi-isomorphisms. Moreover support of $N/I^sN$ is contained in $V(I)$. So there are the following quasi-isomorphisms
\[
F_{\cdot}(N/I^sN)\otimes_{R}\Check{C}_{\underline{y}}\otimes_{R} M \qism F_{\cdot}(N/I^sN)\otimes_{R} \Check{C}_{\underline{y}}\otimes_{R} E^{\cdot}_R(M), \text { and }
\]
\[
F_{\cdot}(N/I^sN)\otimes_R \Check{C}_{\underline{y}} \qism N/I^sN\otimes_R \Check{C}_{\underline{y}}\cong N/I^sN.
\]
Let $L_{\cdot}^R$ denote a free resolution of $M$. Then it follows that the morphism of complexes $F_{\cdot}(N/I^sN)\otimes_{R} \Check{C}_{\underline{y}}\otimes_{R}L_{\cdot}^R\to F_{\cdot}(N/I^sN)\otimes_{R} \Check{C}_{\underline{y}}\otimes_{R}M$ is a homological isomorphism. By the above remark we get the following quasi-isomorphism
\[
F_{\cdot}(N/I^sN)\otimes_{R} \Check{C}_{\underline{y}}\otimes_{R}L_{\cdot}^R\qism N/I^sN\otimes_{R}\Check{C}_{\underline{y}}\otimes_R L_{\cdot}^R\cong N/I^sN\otimes_{R} L_{\cdot}^R
\]
Consequently we get that $H^i(Y)\cong H^i(N/I^sN\otimes_{R}L_{\cdot}^R)\cong \Tor_{-i}^R(N/I^sN,M)$ for all $i\in \mathbb{Z}$ and for all $s\in \mathbb{N}$. By Hom-Tensor Duality (see Lemma \ref{445}) it implies that $H^i(X)\cong \Ext^i_R(N/I^sN,D(M))$ for all $i\in \mathbb{Z}$ and for all $s\in \mathbb{N}$. Then by cohomology sequence there are exact sequences
\begin{equation}\label{1b}
\Ext^i_R(N/I^sN,D(C^{\cdot}_M(I)))\to \Ext^i_R(N/I^sN,D(M))\to \Ext^{i+c}_R(N/I^sN,D(H^c_I(M)))
\end{equation}
for all $i\in \mathbb{Z}$ and for all $s\in \mathbb{N}$. Take the direct limits of it. Then we can easily obtained the natural homomorphisms of $(ii)$ as follows
\[
H^i_{I}(N,D(M))\to H^{i+c}_{I}(N,D(H^c_I(M)))
\]
for all $i\in \mathbb{Z}$. Moreover note that
\[
\lim\limits_{\longrightarrow}\Ext^i_R(N/I^sN,D(C^{\cdot}_M(I)))= H^i_{I}(N,D(C^{\cdot}_M(I)))
\]
for all $i\in \mathbb{Z}$. Hence the isomorphism in $(ii)$ can be easily obtained.
\end{proof}

As a consequence of Proposition \ref{2.3} and Theorem \ref{2} we get the following result.
\begin{corollary}\label{2.11}
Let $M$ be a module over a local ring $R$ and $I$ an ideal with $\grade(I,M)=c.$ If $N$ is a finitely generated $R$-module then the following are true:
\begin{itemize}
\item[(i)] There are the natural homomorphisms
\[
U_{i+c}^I(N, D(M))\to U_{i}^I(N, D(H^c_I(M)))
\]
for all $i\in \mathbb{Z}$. Moreover these are isomorphisms if and only if
$U_i^{I}(N,D(C^{\cdot}_M(I)))=0$ for all $i\in \mathbb{Z}.$

\item[(ii)] There are the natural homomorphisms
\[
L_{i+c}\Lambda^I(N, D(M))\to L_{i}\Lambda^I(N, D(H^c_I(M)))
\]
for all $i\in \mathbb{Z}$. Moreover these are isomorphisms if and only if
$L_i\Lambda^{I}(N,D(C^{\cdot}_M(I)))=0$ for all $i\in \mathbb{Z}.$

\item[(iii)] Suppose in addition that $M$ is Artinian. Then there are the natural homomorphisms
\[
U_{i+c}^I(N, D(D(H^c_I(M))))\to U_{i}^I(N, M)
\]
for all $i\in \mathbb{Z}$. Moreover these are isomorphisms if and only if
$U_i^{I}(N,D(D(C^{\cdot}_M(I))))=0$ for all $i\in \mathbb{Z}.$

\item[(iv)] There are the natural homomorphisms
\[
L_{i+c}\Lambda^I(N, D(D(H^c_I(M))))\to L_{i}\Lambda^I(N, M)
\]
for all $i\in \mathbb{Z}$. Moreover these are isomorphisms if and only if
$L_i\Lambda^{I}(N,D(D(C^{\cdot}_M(I))))=0$ for all $i\in \mathbb{Z}.$
\end{itemize}
\end{corollary}

\begin{proof}
It is obvious in view of Proposition \ref{2.3} and Theorem \ref{2}.
\end{proof}

\begin{proposition}\label{1}
Let $M$ be an module over an arbitrary ring $R$. Suppose that $\grade(I,M)=c$ for an ideal $I$. Then for any $R$-module $N$ we have:
\begin{itemize}
\item[(i)] There are the natural homomorphisms
\[
U_{i+c}^I(N,H^c_I(M))\to U_i^I(N,M)
\]
for all $i\in \mathbb{Z}$.

\item[(ii)] Suppose that $R$ is a local ring and $M$ is an Artinian $R$-module. Further assume that $N$ is finitely generated. Then there are the natural homomorphisms
\[
L_{i+c}\Lambda^I(N, H^c_I(M))\to L_{i}\Lambda^I(N, M)
\]
for all $i\in \mathbb{Z}$.
\end{itemize}
\end{proposition}

\begin{proof}
First of all we prove the statement in $(i)$. For this let $F_{\cdot}^R(N/I^sN)$ be a free resolution of $N/I^sN$. Then tensoring with $F_{\cdot}(N/I^sN)$ to the truncation complex induces the following short exact sequence of complexes
\[
0\to (F_{\cdot}(N/I^sN)\otimes_{R} H^c_I(M))[-c]\to
F_{\cdot}(N/I^sN)\otimes_{R}\Gamma_I(E^{\cdot}_R(M))\to F_{\cdot}(N/I^sN)\otimes_{R} C^{\cdot}_M(I)\to 0
\]
By the proof of Theorem \ref{2}$(ii)$ the homologies of $Y= F_{\cdot}(N/I^sN)\otimes_{R}\Gamma_I(E^{\cdot}_R(M))$ are:
\[
H^i(Y)\cong \Tor_{-i}^R(N/I^sN,M)
\]
for all $i\in \mathbb{Z}$. Then the homology sequence induces the following exact sequence
\begin{equation}\label{1c}
\Tor_{c-i}^R(N/I^sN,H^c_I(M))\to \Tor_{-i}^R(N/I^sN,M)\to \Tor_{-i}^R(N/I^sN,C^{\cdot}_M(I))
\end{equation}
for all $i\in \mathbb{Z}$. As $-i$ varies over $\mathbb{Z}$, we can replace it with $i$. By passing to the inverse limits provides the natural homomorphisms of the statement in $(i)$.

Now suppose that $R$ is local and $M$ is Artinian. Moreover let $N$ be finitely generated. Note that Lemma \ref{1.3} induces the natural homomorphism
\[
L_{i+c}\Lambda^I(N, H^c_I(M))\to U_{i+c}^I(N,H^c_I(M))
\]
for all $i\in \mathbb{Z}$. Moreover $U_{i}^I(N,M)\cong L_{i}\Lambda^I(N, M)$, see Proposition \ref{2.3}$(ii)$, for all $i\in \mathbb{Z}$. Then the claim in $(ii)$ is easy by virtue of $(i)$.

\end{proof}

In the next Corollary we will relate the surjectivity and injectivity of the natural homomorphisms of Theorem \ref{2} and Corollary \ref{2.11}.
\begin{corollary}\label{02}
Let $N$ be a finitely generated module over a local ring $R$. Let $M$ be an $R$-module and $I$ an ideal such that $c=\grade(I,M)$. Then for each fixed $i\in \mathbb{Z}$ we have:
\begin{itemize}
\item[(1)] The following conditions are equivalent:
\begin{itemize}
\item[(i)] The natural homomorphism
\[
H^{i}_{I}(N,H^c_{I}(M))\to H^{i+c}_{I}(N,M)
\]
is injective ( resp. surjective).

\item[(ii)] The natural homomorphism
\[
U_{i+c}^I(N, D(M))\to U_{i}^I(N, D(H^c_I(M)))
\]
is surjective ( resp. injective).

\item[(iii)] The natural homomorphism
\[
L_{i+c}\Lambda^I(N, D(M))\to L_{i}\Lambda^I(N, D(H^c_I(M)))
\]
is surjective ( resp. injective).
\end{itemize}
\item[(2)] If in addition $M$ is an Artinian $R$-module. Then the following conditions are equivalent:
\begin{itemize}
\item[(i)] The natural homomorphism
\[
H^i_{I}(N,D(M))\to H^{i+c}_{I}(N,D(H^c_I(M)))
\]
is injective ( resp. surjective).
\item[(ii)] The natural homomorphism
\[
U_{i+c}^I(N, D(D(H^c_I(M))))\to U_{i}^I(N, M)
\]
is surjective ( resp. injective).
\item[(iii)] The natural homomorphism
\[
L_{i+c}\Lambda^I(N, D(D(H^c_I(M))))\to L_{i}\Lambda^I(N, M)
\]
is surjective ( resp. injective).
\end{itemize}
\end{itemize}
\end{corollary}

\begin{proof}
By Matlis duality it is obvious from Proposition \ref{2.3}, Theorem \ref{2} and Corollary \ref{2.11}.
\end{proof}

\section{Vanishing and non-vanishing properties}

In this section we will describe the vanishing and non-vanishing results of $L_i\Lambda^I (N,M)$. Moreover we will prove that our natural homomorphisms described in the previous section are isomorphisms under the additional assumption of $M$ is cohomologically complete intersection with respect to $I$. Moreover we are succeeded to prove some equivalent conditions such that $M$ is cohomologically complete intersection with respect to the pair $(N,I)$. At the end of this section there is a characterization of grade and co-grade. As a first step we have the following vanishing result:

\begin{corollary}\label{001}
Let $I$ be an ideal over a local ring $R$. Let $M$ be an $R$-module. Then for any finitely generated $R$-module $N$ we have:
\begin{itemize}
\item[(i)] $H^i_{I}(N,M)=0$ if and only if $U^I_i(N,D(M))=0$ if and only if $L_i\Lambda^I (N,D(M))=0$ for each fixed $i\in \mathbb{Z}$.

\item[(ii)] If $M$ is Artinian. Then $H^i_{I}(N,D(M))=0$ if and only if $U^I_i(N,M)=0$ if and only if $L_i\Lambda^I (N,M)=0$ for each fixed $i\in \mathbb{Z}$.

\end{itemize}
\end{corollary}

\begin{proof}
It is an immediate consequence of Proposition \ref{2.3}.
\end{proof}

\begin{corollary}\label{0.01}
Let $N$ be a finitely generated module over a local ring $R$. Let $M$ be an $R$-module and $I$ an ideal such that $c=\grade(I,M)$. Then the following conditions are equivalent:
\begin{itemize}
\item[(i)] $M$ is cohomologically complete intersection with respect to $(N,I)$.

\item[(ii)] $U^I_i(N,D(M))=0$ for all $i\neq c.$

\item[(iii)] $L_i\Lambda^I (N,D(M))=0$ for all $i\neq c.$

\end{itemize}
\end{corollary}

\begin{proof}
For the proof see Corollary \ref{001}$(i)$.
\end{proof}

For the rest of the paper we need the theory of spectral sequences. For this purpose we refer to \cite{b}, \cite{r} and \cite{w}.
\begin{proposition}\label{2.1}
Let $R$ be any ring, $I$ an ideal of $R$ and $N$ a finitely generated $R$-module. Suppose that $M$ is an $R$-module with $\grade(I,M)=c$. Then the following are true:
\begin{itemize}
\item[(i)] The natural homomorphism
\[
H^0_I(N,H^c_I(M))\to H^{c}_I(N, M)
\]
is an isomorphism and $H^i_{I}(N,M)=H^{i-c}_{I}(N,H^c_I(M))= 0$ for all $i< c$.

\item[(ii)] Suppose now that $R$ is local. Then the natural homomorphism
\[
U^I_c(N,D(M))\to U^I_{0}(N,D(H^c_I(M))
\]
is an isomorphism and $U^I_i(N,D(M))= U^I_{i-c}(N,D(H^c_I(M)))=0$ for all $i< c$.

\item[(iii)] The natural homomorphism
\[
L_c\Lambda^I (N,D(M))\to L_{0}\Lambda^I (N,D(H^c_I(M)))
\]
is an isomorphism and $L_i\Lambda^I (N,D(M))=L_{i-c}\Lambda^I (N,D(H^c_I(M)))=0$ for all $i< c$.
\end{itemize}
\end{proposition}

\begin{proof}
By Corollaries \ref{02}$(1)$ and \ref{001}$(1)$ we will only prove the claim in $(i)$. Let $E_{\cdot}^R$ be an injective resolution of the truncation complex $C^{\cdot}_M(I)$. Then by definition of the truncation complex it follows that $H^i(E_{\cdot}^R)=0$ for all $i\leq c$. Moreover note that
\[
\Ext_R^i(\cdot,C^{\cdot}_M(I))\cong H^i(\Hom_R(\cdot, E_{\cdot}^R))
\]
for all $i\in \mathbb{Z}$ (see \cite[p. 331]{r}). Now let $s\in \mathbb{N}$ be a fixed natural number and consider the following spectral sequence
\[
E^{p,q}_2:=\Ext^p_R(N/I^sN,H^q_I(E_{\cdot}^R))\Longrightarrow E^{p+q}_\infty= H^{p+q}(\Hom_R(N/I^sN,E_{\cdot}^R))
\]
(see \cite[Theorem 11.38]{r}). Let $p + q\leq c$. Note that for $p\geq 0$ we have $q\leq c$. Therefore it turns out that
\[
H^{p+q}(\Hom_R(N/I^sN,E_{\cdot}^R))=0 \text { for all $p + q\leq c$}.
\]
It implies that $H^i(\Hom_R(N/I^sN, E_{\cdot}^R))=0$ for all $i\leq c$ and for all $s\in \mathbb{N}$ as a consequence of the spectral sequence. So we have
\[
H^i_I(N,C^{\cdot}_M(I))=\lim\limits_{\longrightarrow}\Ext_R^i(N/I^sN,C^{\cdot}_M(I))= 0
\]
for all $i\leq c$. Moreover by \cite[Proposition 5.5]{b1} it follows that $H^i_{I}(N,M)=0$ for all $i< c$ and  $H^c_{I}(N,M)\neq 0$. After taking the direct limit the exact sequence \ref{1a} implies that there is an isomorphism
\[
H^0_{I}(N,H^c_I(M))\to H^c_{I}(N,M)
\]
and $H^{i-c}_{I}(N,H^c_I(M))=0$ for all $i< c$. This finish the proof of Proposition.
\end{proof}

\begin{corollary}\label{2.2}
Let $I$ be an ideal of an arbitrary ring $R$ and $M$ a cohomologically complete intersection module with respect to $I$. Then for any $R$-module $N$ the following are true:
\begin{itemize}
\item[(i)] The natural homomorphism
\[
H^{i}_{I}(N,H^c_{I}(M))\to H^{i+c}_{I}(N,M)
\]
is an isomorphism for all $i\in \mathbb{Z}$.

\item[(ii)] The natural homomorphism
\[
U_{i+c}^I(N,H^c_I(M))\to U_{i}^I(N,M)
\]
is injective for all $i\in \mathbb{Z}$.

\item[(iii)] Suppose in addition that $R$ is local. Then the natural homomorphism
\[
H^i_{I}(N,D(M))\to H^{i+c}_{I}(N,D(H^c_I(M)))
\]
is an isomorphism for all $i\in \mathbb{Z}$.

\item[(iv)] Assume in addition that $N$ is finitely generated. Then the natural homomorphism
\[
U_{i+c}^I(N, D(M))\to U_{i}^I(N, D(H^c_I(M)))
\]
is an isomorphism for all $i\in \mathbb{Z}$.

\item[(v)] The natural homomorphism
\[
L_{i+c}\Lambda^I(N, D(M))\to L_{i}\Lambda^I(N, D(H^c_I(M)))
\]
is an isomorphism for all $i\in \mathbb{Z}$.

\item[(vi)] Further assume that $M$ is Artinian. Then the natural homomorphism
\[
U_{i+c}^I(N, D(D(H^c_I(M))))\to U_{i}^I(N, M)
\]
is an isomorphism for all $i\in \mathbb{Z}$.

\item[(vii)] The natural homomorphism
\[
L_{i+c}\Lambda^I(N, D(D(H^c_I(M))))\to L_{i}\Lambda^I(N, M)
\]
is an isomorphism for all $i\in \mathbb{Z}$.
\end{itemize}
\end{corollary}

\begin{proof}
Since $H^i_I(M)= 0$ for all $i\neq c$. Then by definition of the truncation complex $H^i(C^{\cdot}_R(I))=0$ for all $i\in \mathbb{Z}$. So the complex $C^{\cdot}_R(I)$ is exact. Let $F_{\cdot}(N/I^sN)$ be a free resolution of $N/I^sN$. Then the complex $\Hom_R(F_{\cdot}(N/I^sN),C^{\cdot}_R(I))$ is exact too, for each $s\in \mathbb{N}$. It implies that
\[
H^i_I(N,C^{\cdot}_M(I))=\lim\limits_{\longrightarrow}\Ext_R^i(N/I^sN,C^{\cdot}_M(I))= 0
\]
for all $i\in \mathbb{Z}$. So by Theorem \ref{2}$(i)$ this completes the proof of $(i)$. Moreover the statements in $(iv)$ and $(v)$ is a consequence of Corollary \ref{02}$(1)$. By the similar arguments with the help of Theorem \ref{2}$(ii)$ and Corollary \ref{02}$(2)$ we can prove the isomorphisms in $(iii)$, $(vi)$ and $(vii)$.

Note that by the above remarks the complex $F_{\cdot}(N/I^sN)\otimes_R C^{\cdot}_R(I)$ is exact for each $s\in \mathbb{N}$ . It proves the injectivity of $(ii)$ by virtue of the exact sequence \ref{1c}.
\end{proof}

Now we are able to prove our main results.

\begin{theorem}\label{2.32}
Let $0\neq M$ be a module over an arbitrary ring $R$. Let $I$ be an ideal and $c=\grade(I,M)$. Then the following conditions are equivalent:
\begin{itemize}
\item[(i)] $M$ is cohomologically complete intersection with respect to $I$.

\item[(ii)] For any $R$-module $N$ the natural homomorphism
\[
H^{i}_{I}(N,H^c_{I}(M))\to H^{i+c}_{I}(N,M)
\]
is an isomorphism for all $i\in \mathbb{Z}$.
\end{itemize}
Assume in addition that $R$ is local. Then the above conditions are equivalent to the following:
\begin{itemize}
\item[(iii)] For any finitely generated $R$-module $N$ the natural homomorphism
\[
U_{i+c}^I(N, D(M))\to U_{i}^I(N, D(H^c_I(M)))
\]
is an isomorphism for all $i\in \mathbb{Z}$.

\item[(iv)] For any finitely generated $R$-module $N$ the natural homomorphism
\[
L_{i+c}\Lambda^I(N, D(M))\to L_{i}\Lambda^I(N, D(H^c_I(M)))
\]
is an isomorphism for all $i\in \mathbb{Z}$.
\end{itemize}
\end{theorem}

\begin{proof}
By Corollaries \ref{02}$(1)$ and \ref{2.2}$(i)$ we will only prove that $(ii)$ implies $(i)$. Suppose that $(ii)$ is true. Then substitute $N=R$ we get that $H^i_I(C^{\cdot}_M(I))= 0$ for all $i\in \mathbb{Z}$ (see Theorem \ref{2}$(i)$). By definition of the truncation complex and \cite[Lemma 2.5]{waqas2} there are the isomorphisms
\[
 H^i_{I}(C^{\cdot}_M(I))\cong H^i(C^{\cdot}_M(I)) \cong H^i_I(M) \text{ for } i> c.
\]
This proves that $M$ is cohomologically complete intersection with respect to $I$.
\end{proof}

\begin{theorem}
Let $0\neq M$ be a module over a local ring $R$. Let $I$ be an ideal and $c=\grade(I,M)$. Then the following conditions are equivalent:
\begin{itemize}
\item[(i)] $M$ is cohomologically complete intersection with respect to $I$.

\item[(ii)] For any $R$-module $N$ the natural homomorphism
\[
H^i_{I}(N,D(M))\to H^{i+c}_{I}(N,D(H^c_I(M)))
\]
is an isomorphism for all $i\in \mathbb{Z}$.
\end{itemize}
Assume in addition that $M$ is Artinian. Then the above conditions are equivalent to the following:
\begin{itemize}
\item[(iii)] For any finitely generated $R$-module $N$ the natural homomorphism
\[
U_{i+c}^I(N, D(D(H^c_I(M))))\to U_{i}^I(N, M)
\]
is an isomorphism for all $i\in \mathbb{Z}$.

\item[(vi)] For any finitely generated $R$-module $N$ the natural homomorphism
\[
L_{i+c}\Lambda^I(N, D(D(H^c_I(M))))\to L_{i}\Lambda^I(N, M)
\]
is an isomorphism for all $i\in \mathbb{Z}$.
\end{itemize}

\end{theorem}

\begin{proof}
Note that it will be enough to prove that $(ii)$ implies $(i)$ (see Corollaries \ref{02}$(2)$ and \ref{2.2}$(iii)$). To do this let us assume that $(ii)$ is true. Then substitute $N=R$ we get that $H^i_I(D(C^{\cdot}_M(I)))= 0$ for all $i\in \mathbb{Z}$ (see Theorem \ref{2}$(ii)$). We claim that $H^i_I(C^{\cdot}_M(I))= 0$ for all $i\in \mathbb{Z}$. Note that by \cite[Theorem 2.2]{v} it will be enough to prove that $\Ext^{i}_R(R/I,C^{\cdot}_M(I))= 0$ for all $i\in \mathbb{Z}$.

To do this let $\Check{C}_{\underline{x}}$ be the \v{C}ech complex with respect to $\underline{x}= x_1, \ldots ,x_s\in I$ such that $\Rad I= \Rad(\underline{x})R$. Then it implies that $\Check{C}_{\underline{x}}\otimes_{R} D(C^{\cdot}_M(I))$ is an exact complex. Here we use that $H^i_I(D(C^{\cdot}_M(I)))= 0$ for all $i\in \mathbb{Z}$. Suppose that $L^{\cdot}_R$ denote a free resolution of $D(C^{\cdot}_M(I)).$ Let $X:= R/I\otimes_{R} L^{\cdot}_R$ then there is an isomorphism
\[
\Tor_{-i}^R(R/I, D(C^{\cdot}_M(I)))\cong H^i(X)
\]
for all $i\in \mathbb{Z}$. Since the support of each module of $X$ is in $V(I)$. It follows that there is an isomorphism of complexes
\[
\Check{C}_{\underline{x}}\otimes_{R} X\cong X.
\]
Let $F_{\cdot}^R$ be a free resolution of a finitely generated $R$-module $R/I$. By the above arguments the following complex is exact
\[
Y:=F_{\cdot}^R\otimes_{R} \Check{C}_{\underline{x}}\otimes_{R} D(C^{\cdot}_M(I)).
\]
Moreover $F_{\cdot}^R$ is a right bounded complex of finitely generated free $R$-modules and $\Check{C}_{\underline{x}}$ is a bounded complex of flat $R$-modules. So $Y$ is quasi-isomorphic to $\Check{C}_{\underline{x}}\otimes_{R} F_{\cdot}^R\otimes_{R} L^{\cdot}_R$. So it is homologically trivial. Note that the morphism of complexes
\[
\Check{C}_{\underline{x}}\otimes_{R} F_{\cdot}^R\otimes_{R} L^{\cdot}_R\to \Check{C}_{\underline{x}}\otimes_{R} R/I\otimes_{R} L^{\cdot}_R= \Check{C}_{\underline{x}}\otimes_{R} X
\]
induces an isomorphism in cohomologies. It follows that the complex $\Check{C}_{\underline{x}}\otimes_{R} X$ is homologically trivial. Therefore $\Tor_{i}^R(R/I, D(C^{\cdot}_M(I)))=0$ for all $i\in \mathbb{Z}$. Then by Lemma \ref{445} we get that
\[
\Ext^{i}_R(R/I,C^{\cdot}_M(I))=0 \text { for all $i \in \mathbb Z$.}
\]
This proves the claim that $H^i_I(C^{\cdot}_M(I))= 0$ for all $i\in \mathbb{Z}$. Hence by the proof of Theorem \ref{2.32} $M$ is cohomologically complete intersection with respect to $I$.
\end{proof}

In the following, as an application of our results, we will prove some results on grade and cograde.

\begin{corollary}\label{31}
Let $I$ be an ideal of an arbitrary ring $R$. Let $M$ and $N$ be finitely generated modules over $R$ with $c=\grade(I,M)$. Then the following are true:
\begin{itemize}
\item[(i)] $c=\inf\{i\in \mathbb{N}: H^i_{I}(N,M)\neq 0\}=$ $\inf\{i\in \mathbb{N}: H^{i-c}_{I}(N,H^c_I(M))\neq 0\}$.

\item[(ii)] Suppose now that $R$ is local. Then we have
\[
c=\inf\{i\in \mathbb{N}: U^I_i(N,D(M))\neq 0\}=\inf\{i\in \mathbb{N}: U^I_{i-c}(N,D(H^c_I(M))\neq 0\}
\]
$=\inf\{i\in \mathbb{N}: L_i\Lambda^I (N,D(M))\neq 0\}=$ $\inf\{i\in \mathbb{N}: L_{i-c}\Lambda^I (N,D(H^c_I(M)))\neq 0\}$.

\item[(iii)] If in addition $M$ is Artinian and set $t:=\Cograde_M(N/IN)$. Then we have
\[
t=\inf\{i\in \mathbb{N}: L_i\Lambda^I(N,M)\neq 0\}=\inf\{i\in \mathbb{N}: H^i_{I}(N,D(M))\neq 0\}.
\]
\end{itemize}
\end{corollary}

\begin{proof}
It is well-known that
\[
c=\inf\{i\in \mathbb{N}: H^i_{I}(N,M)\neq 0\} \text { and } t=\inf\{i\in \mathbb{N}: U^I_i(N,M)\neq 0\}
\]
(see \cite[Theorem 2.3]{s}, \cite[Proposition 5.5]{b1} and \cite[Theorem 4.2]{b2}). Then all the claims of the Corollary can be easily seen with the help of Corollary \ref{001}$(ii)$ and Proposition \ref{2.1}.
\end{proof}


\end{document}